\documentclass[11pt]{amsart}
\usepackage[utf8]{inputenc}
\usepackage[english]{babel}

\usepackage[square,sort,comma,numbers]{natbib}
\usepackage{pinlabel}

\usepackage[dvipsnames]{xcolor}
\usepackage{float}
\usepackage{pinlabel}
\usepackage{graphicx}
\usepackage{tikz}
\usetikzlibrary{arrows}
\usetikzlibrary{matrix,arrows,decorations.pathmorphing}

\usepackage{amsmath}
\usepackage{amssymb}
\usepackage{amsthm}

\usepackage[all]{xy}
\usepackage[mathcal]{eucal}
\usepackage{verbatim}\usepackage{fullpage} \usepackage{wrapfig}
\usepackage{lipsum}
\usepackage[all]{xy}
\theoremstyle{plain}
\newtheorem{thm}{Theorem}[section]

\newtheorem{lem}[thm]{Lemma}

\newtheorem{prop}[thm]{Proposition}

\newtheorem{claim}[thm]{Claim}

\newtheorem{qu}[thm]{Problem}
\theoremstyle{definition}
\newtheorem{defn}[thm]{Definition}
\newtheorem{fact}[thm]{Fact}

\theoremstyle{remark}
\newtheorem*{rem}{Remark}

\newcommand{\nc}{\newcommand}
\nc{\dmo}{\DeclareMathOperator}

\DeclareMathOperator{\Diff}{Diff}

\DeclareMathOperator{\Homeo}{Homeo}
\DeclareMathOperator{\DD}{\mathbb{D}}

\nc{\para}[1]{\medskip\noindent\textbf{#1.}}
\title{Non-realizability of the pure braid group as area-preserving homeomorphisms}

\author{Lei Chen}
\address{\newline Department of Mathematics   \newline California Institute of Technology   \newline Pasadena, CA 91125,  USA }
\email{chenlei@caltech.edu}
\begin{document}
 \bibliographystyle{alpha}
 
\maketitle
\begin{abstract}
Let $\Homeo_+(D^2_n)$ be the group of orientation-preserving homeomorphisms of $D^2$ fixing the boundary pointwise and $n$ marked points as a set. Nielsen realization problem for the braid group asks whether the natural projection $p_n:\Homeo_+(D^2_n)\to B_n:=\pi_0(\Homeo_+(D^2_n))$ has a section over subgroups of $B_n$. All of the previous methods either use torsions or Thurston stability, which do not apply to the pure braid group $PB_n$, the subgroup of $B_n$ that fixes $n$ marked points pointwise.
In this paper, we show that the pure braid group has no realization inside the area-preserving homeomorphisms using rotation numbers.
\end{abstract}

\section{Introduction}
Denote by $D^2$ the 2-dimensional disk. Let $\Homeo_+(D^2_n)$ be the group of orientation-preserving homeomorphisms of $D^2$ fixing the boundary pointwise and $n$ marked points as a set. Denote by $B_n:=\pi_0(\Homeo_+(D^2_n))$. The Nielsen realization problem for $B_n$ asks whether the natural projection
\[
p_n: \Homeo_+(D^2_n)\to B_n
\]
has a section over subgroups of $B_n$. For the whole group $B_n$, this question has several previous results. Salter--Tshishiku \cite{ST} uses Thurston stability to show that $B_n$ has no realization in $\Diff_+(D^2_n)$ and the author \cite{Chen} uses ``hidden torsions" and Markovic's machinery \cite{Mar} to show that $B_n$ has no realization in $\Homeo_+(D^2_n)$. Let $PB_n<B_n$ be the subgroup that preserves $n$ marked points pointwise. The Nielsen realization problem for $PB_n$ is widely open since the two methods in \cite{ST} and \cite{Chen} fail to work and has no hope to repair. The following question is asked by \cite[Question 3.12]{MT} and \cite[Remark 1.4]{ST}.
\begin{qu}[Realization of pure braid group]
Does $PB_n$ have realization as diffeomorphisms or homeomorphisms? In other words, does $p_n$ have sections over $PB_n$?
\end{qu}
Denote by $\Homeo_+^a(D^2_n)$  the group of orientation-preserving, area-preserving homeomorphisms of $D^2$ fixing the boundary pointwise and $n$ marked points as a set.
In this paper, we make a progress proving  the following result.
\begin{thm}\label{main}
The pure braid group cannot be realized as area-preserving homeomorphisms on $D^2_n$ for $n\ge 9$. In other words, the natural projection $p_n^a:\Homeo_+^a(D^2_n)\to B_n$ has no sections over $PB_n$. 
\end{thm}
We remark that the Nielsen realization problem is closely related to the existence of flat structures on a surface bundle. We refer the reader to \cite{MT} for more history and background.

\para{Comparing with the method in \cite{ChenMark}}
The novelty of this paper is to provide a different ending towards \cite{ChenMark}. The original ending is to use the fact that certain Dehn twist is a product of commutator in its centralizer. However, such structure does not hold in $PB_n$. Instead, we prove a stronger dynamical property about Dehn twists about non-separating curves. In the beginning of Section 4, we present an outline of the proof. Since this paper has a lot of overlap with \cite{ChenMark}, we omit or sketch many proofs to reduce redundancy. 

\para{Organization of the paper}
\begin{itemize}
\item In Section 2, we discuss rotation numbers;
\item In Section 3, we discuss the pure braid group and the minimal decomposition theory;
\item In Section 4, we give an outline of the proof and finish the argument.
\end{itemize}

\para{Acknowledgement}
The author would like to thank Vlad Markovic for helpful discussion.

\section{Rotation numbers of annulus homeomorphisms}
In this section, we discuss the properties of rotation numbers on annuli. 

\subsection{Rotation number of an area-preserving homeomorphism of an annulus}
Firstly, we define the rotation number for geometric annuli. Let 
\[
N=N(r)=\{w \in \mathbb{C}:\frac{1}{r}< |w|<r \}
\]
be the geometric annulus in the complex plane $\mathbb{C}$. Denote the geometric strip in $\mathbb{C}$ by 
\[
P=P(r)=\{x+iy=z\in \mathbb{C}:|y|<\frac{\log r}{2\pi}\}.
\] 
The map $\pi(z)=e^{2\pi iz}$ is a holomorphic covering map  $\pi: P\to N$. The deck transformation on $P$ is $T(x,y)=(x+1,y)$. 

Denote by $p_1:P\to \mathbb{R}$ the projection to the $x$-coordinate, and by $\Homeo_+(N)$ the group of homeomorphisms of $N$ that preserves orientation and the two ends.  Fix $f\in \Homeo_+(N)$, and $x\in N$, and let $\widetilde{x}\in P$ and $\widetilde{f}\in \Homeo_+(P)$ denote lifts of $x$ and $f$ respectively. We define the translation number of the lift $\widetilde{f}$ at $\widetilde{x}$ by
\begin{equation}\label{rot-0}
\rho(\widetilde{f},\widetilde{x},P)=\lim_{n\to \infty} (p_1(\widetilde{f}^n(\widetilde{x}))-p_1(\widetilde{x}))/n.
\end{equation}
The rotation number of $f$ at $x$ is then defined as
\begin{equation}\label{rot}
\rho(f,x,N)=\rho(\widetilde{f},\widetilde{x},P)  \,\,\,\,\,\, \,   \,\,\,\,\,\, \, (\text{mod 1}).
\end{equation}

The rotation number is not defined everywhere (see, e.g., \cite{Franks} for more background on rotation numbers). The closed annulus  $N_c$ is 
\[
N_c=\{\omega \in \mathbb{C}:\frac{1}{r}\le |\omega|\le r \},
\]
For $f\in \Homeo_+(N_c)$, the rotation and translation numbers are defined analogously.

Let $A$ be an open annulus embedded in a Riemann surface (in particular this endows $A$ with the complex structure). By the Riemann mapping theorem, there is a unique $N(r)=N$ and a conformal map $u_A: A\to N$. For any $f\in \Homeo_+(A)$ (the group of end-preserving homeomorphisms), we 
define the rotation number of $f$ on $A$ by
\[\rho(f,x,A):=\rho(g,u_A(x), N),\]  where $g=u_A\circ f\circ u_A^{-1}$.

We have the following theorems of Poincar\'e-Birkhoff and Handel about rotation numbers \cite{Handel} (See also Franks \cite{Franks}). 
\begin{thm}[Properties of rotation numbers]\label{rotationproperty}
If $f: N_c \to N_c$ is an orientation preserving, boundary component preserving, area-preserving homeomorphism and $\widetilde{f}: P_c \to P_c$ is any lift, then:
\begin{itemize}
\item (Handel) The translation set \[
R(\widetilde{f})=\bigcup_{\widetilde{x} \in P_c} \rho(\widetilde{f},\widetilde{x},P_c)\] is a closed interval.
\item (Poincar\'e-Birkhoff) If $r\in R(\widetilde{f})$ is rational, then there exists a periodic orbit of $f$ realizing the rotation number $r$ mod $1$. 
\end{itemize}
\end{thm}

\subsection{Separators and its properties}
We let $A$ continue to denote an open annulus embedded in a Riemann surface. Then $A$ has two ends and we choose  one of them to be the left end and the other one to be the right end. We call a  subset $X\subset A$ \emph{separating} (or essential) if every 
arc $\gamma \subset A$ which connects the two ends of $A$  must intersect $X$. 

\begin{defn}[Separator]
We call a subset $M\subset A$ a \emph{separator} if $M$ is compact, connected and separating.
\end{defn}
The complement of $M$ in $A$ is a disjoint union of open sets. We have the following lemma.
\begin{lem}\label{separator}
Let $M$ be a separator. Then there are exactly two connected components $A_L(M)$ and $A_R(M)$ of $A-M$ which are open annuli homotopic to $A$ and with the property that $A_L(M)$ contains the left end of $A$ and $A_R(M)$ contains the right end of $A$. All other components of $A-M$ are simply connected.
\end{lem}
\begin{proof}
We compactify the annulus $A$ by adding  points $p_L$ and $p_R$ to the corresponding ends of $A$. The compactifications is a two sphere $S^2$. Moreover, $M$ is a compact and connected subset of $S^2 -\{p_L,p_R\}$.

Now, we observe that every component of  $S^2-M$ is  simply connected. Denote by $\Omega_L$ and $\Omega_R$ the  connected components of $S^2-M$  
containing $p_L$ and $p_R$  respectively. Since $M$ is separating we conclude that these are two different components. We define $A_L(M)=\Omega_L-p_L$ and $A_R(M)=\Omega_R-p_R$. It is easy to verify that these are required annuli.
\end{proof}

We now prove another property of a separator. Let $\pi: \widetilde{A}\to A$ be the universal cover.

\begin{prop}\label{sepprop}
Let $M \subset A$ be a compact domain with smooth boundary. Then $\pi^{-1}(M)$ is connected.
\end{prop}

\begin{proof}
Since $M$ is a compact domain with boundary which separates the two ends of $A$, we can find a circle $\gamma\subset M$ which is essential in $A$ (i.e. $\gamma$ is a separator itself) (note that $M$ has only finitely many boundary components).  Denote by $T$ the deck transformation of $\widetilde{A}$. Thus, the lift $\pi^{-1}(\gamma)$ is a $T$-invariant, connected subset of $\widetilde{A}$. Let $C$ be the component of $\pi^{-1}(M)$ which contains $\pi^{-1}(\gamma)$. Then $C$ is $T$ invariant. We show $\pi^{-1}(M)=C$.

Let $p\in M$. Since $M$ is a compact  domain with smooth boundary, we can find an embedded closed arc $\alpha\subset M$ which connects $p$ and $\gamma$. Let $\widetilde{p}$ be a lift of $p$ and let $\widetilde{\alpha}$ be the corresponding lift of $\alpha$ such that $\widetilde{p}$ is one of its endpoints. Then, the other endpoint of  $\widetilde{\alpha}$ is in $\pi^{-1}(\gamma)$, and this shows that  $\widetilde{p} \in C$. This concludes the proof.

\end{proof}

Now we discuss an ordering on the set of separators. 
\begin{prop}\label{ordering}
Suppose $M_1,M_2 \subset A$ are two disjoint separators. Then either $M_1\subset A_L(M_2)$  or  $M_1\subset A_R(M_2)$. Moreover, $M_1\subset A_L(M_2)$ implies  $M_2\subset A_R(M_1)$. 
\end{prop}
\begin{proof}
Since $M_1$ is connected it follows that $M_1$ is a subset of a connected component $C$ of $A-M_2$. Since $C$ is open, we know that there is a neighborhood $N_1$ of $M_1$ with smooth boundary such that $N_1\subset C$ (It is elementary to construct such $N_1$). If $C$ is simply connected, the cover $\pi^{-1}(C)\to C$ is a trivial cover. Let $\widetilde{C}$ be a connected component of $\pi^{-1}(C)$. By Proposition \ref{sepprop}, the set $\pi^{-1}(N_1)$ is connected so it is contained in a single connected component of $\pi^{-1}(C)$. However, this contradicts the fact that  $\pi^{-1}(N_1)$ is also translation invariant. Thus,  either $M_1\subset A_L(M_2)$ or $M_1\subset A_R(M_2)$.

Suppose  $M_1\subset A_L(M_2)$. Then $A_L(M_1)\subset A_L(M_2)$ as well. On the other hand, by the first part of the proposition we already know that either $M_2\subset A_L(M_1)$ or $M_2\subset A_R(M_1)$. If  $M_2\subset A_L(M_1)$, then $A_L(M_2)\subset A_L(M_1)$. This shows that $A_L(M_1)\subset A_L(M_2)$ which implies that $M_2\subset A_L(M_2)$. This is absurd so we must have $M_2\subset A_R(M_1)$.
\end{proof}

\vskip .3cm

\begin{defn}\label{def-ordering}

The inclusion $M_1\subset A_L(M_2)$ is denoted as $M_1<M_2$.

\end{defn}

\vskip .3cm

\subsection{The rotation interval of an annular continuum and prime ends}

Let  $K\subset A$ be a separator (in literature also known as an \emph{essential continuum}). We call $K$ an essential \emph{annular continuum} if $A - K$ has exactly two components. Observe that an essential annular continuum can be expressed  as a decreasing intersection of essential closed topological annuli in $A$.

It is possible to turn any separator $M\subset A$ into  an essential  annular continuum. Let $M$ be a separating connected set. By Lemma \ref{separator}, we know that $A-M$ has exactly two connected annular components $A_L(M)$ and $A_R(M)$, and all other components of $A-M$ are simply connected. We call a simply connected component of $A-M$ a \emph{bubble component}. Then the \emph{annular completion} $K(M)$ of $M$ is defined as the union of $M$ and the corresponding  bubble components of $A-M$.

\begin{prop}
Let $M\subset A$ be a separator. Then the annular completion $K(M)$ is an annular continuum. 
\end{prop}
\begin{proof}
We can again compactify $A$ by adding the points $p_L$ and $p_R$, one at each end. The compactification is the two sphere $S^2$. Then $A_L(M)$ and $A_R(M)$ are two disjoint open discs in $S^2$, and $K(M)=S^2-(A_L(M)\cup A_R(M))$. But the complement of two disjoint open discs in $S^2$ is connected. This proves the proposition.

\end{proof}

Now let $f$ be a homeomorphism of $A$ that leaves an annular continuum $K$ invariant. If $\mu$ is an invariant Borel probability measure supported on $K$,  we define the $\mu$-rotation number 
\[
\sigma(f, \mu)=\int_A\phi d\mu
\]
where $\phi: A \to \mathbb{R}$ is the function which lifts to the function $p_1\circ f-p_1$ on $\widetilde{A}$ (recall that $p_1:\widetilde{A} \to  \mathbb{R}$ is the projection onto the first coordinate).

The set of $f$ invariant Borel probability measures on $K$ is a non empty, convex, and compact set (with respect to the weak topology on the space of measures).
We define the \emph{rotation interval} of $K$
\[
\sigma(f, K)=\{\sigma(f, \mu)|\mu\in M(K)\}
\]
which is a non-empty segment $[\alpha,\beta]$ of $\mathbb{R}$. The interval is non empty because there exists at least one $f$ invariant measure, and it is an interval because the set of $f$ invariant measures is convex.

The following is a classical result of Franks--Le Calvez \cite[Corollary 3.1]{FLC}.
\begin{prop}
If $\sigma(f, K)=\{\alpha\}$, the sequence 
\[
\frac{p_1\circ f^n(x)-p_1(x)}{n}
\]
converges uniformly for $x\in \pi^{-1}(K)$ to the constant function $\alpha$. This implies that points in $K$ all have the rotation number $\alpha$.
\end{prop}

The following  theorem of Franks--Le Calvez \cite[Proposition 5.4]{FLC}  is a generalization of the Poincar\'e-Birkhoff Theorem.
\begin{thm}\label{PB}
If $f$ is area-preserving and $K$ is an annular continuum, then every rational number in $\sigma(f,K)$ is realized by a periodic point in $K$.
\end{thm}

The theory of prime ends is an important tool in the study of 2-dimensional dynamics which can be used to transform a 2-dimensional problem into a 1-dimensional problem. Recall that we assume that $A$ is an open annulus embedded in a Riemann surface $S$.  Suppose that $f$ is a homeomorphism of $S$ which leaves $A$ invariant. Furthermore, let $K\subset A$ be an annular continuum and  suppose that $f$  leaves $K$ invariant. Then both $A_L(K)$ and $A_R(K)$ are $f$ invariant. 

Since $A$ is embedded in $S$, we can define the frontiers of $A$, $A_L(K)$, and $A_R(K)$. By Carath\'eodory's theory of prime ends (see, e.g., \cite[Chapter 15]{Milnor}), the homeomorphism $f$   yields an action on the frontiers of $A_L(K)$ and $A_R(K)$.  Consider the right hand frontier of $A_L(K)$ (the one which is contained in $A$). Then the set of prime ends on this frontier is homeomorphic to the circle, and we denote by  $f_L$ the induced homeomorphism of this circle. Likewise, the set of prime ends on left hand  frontier of $A_R(K)$  is homeomorphic to the circle, and we denote by  $f_R$ the induced homeomorphism this circle. 

The rotation number of a circle homeomorphism   (defined by Equation \eqref{rot}), is well defined everywhere and is the same number for any point on the circle. The rotation numbers of $f_L$ and $f_R$ are called $r_L$ and $r_R$. We refer to them as the left and right prime end rotation numbers of $f$. We have the following theorem of Matsumoto \cite{Mat}.
\begin{thm}[Matsumoto's theorem]\label{Matsumoto}
If $K$ is an annular continuum, then its left and right
prime ends rotation numbers $r_L,r_R$ belong to the rotation interval $\sigma(f, K)$.
\end{thm}

\section{Minimal decompositions and characteristic annuli}

\subsection{Minimal decompositions} 
We recall the theory of minimal decompositions of surface homeomorphisms. This is established in \cite{Mar}. 
Firstly we recall the upper semi-continuous decomposition of a surface; see also Markovic \cite[Definition 2.1]{Mar}. Let $M$ be a surface. 
\begin{defn}[Upper semi-continuous decomposition]
Let $\mathbf{S}$ be a collection of closed, compact, connected subsets of $M$. We say that $\mathbf{S}$ is an upper semi-continuous decomposition of $M$ if the following holds:
\begin{itemize}
\item If $S_1,S_2\in \mathbf{S}$, then $S_1\cap S_2=\emptyset$.
\item If $S\in \mathbf{S}$, then $E$ does not separate $M$; i.e., $M-S$ is connected.
\item We have $M=\bigcup_{S\in \mathbf{S}} S$.
\item If $S_n\in \mathbf{S}, n\in \mathbb{N}$ is a sequence that has the Hausdorff limit equal to $S_0$ then there exists $S\in\mathbf{S}$ such that $S_0\subset S$.
\end{itemize}
\end{defn}
Now we define acyclic sets on a surface.
\begin{defn}[Acyclic sets]
Let $S\subset M$ be a closed, connected subset of $M$ which does not separate $M$. We say that $S$ is \emph{acyclic} if there is a simply connected open set $U\subset M$ such that $S \subset U$ and $U-S$ is homeomorphic to an annulus.
\end{defn}
The simplest examples of acyclic sets are a point, an embedded closed arc and an embedded closed disk in $M$. Let $S\subset M$ be a closed, connected set that does not separate M. Then $S$ is acyclic if and only if there is a lift of $S$ to the universal cover $\widetilde{M}$ of $M$, which is a compact subset of $\widetilde{M}$. 
The following theorem is a classical result called Moore's theorem; see, e.g., \cite[Theorem 2.1]{Mar}. 
\begin{thm}[Moore's theorem]\label{moore}
Let $M$ be a surface and $\mathbf{S}$ be an upper semi-continuous decomposition of $M$ so that every element of $\mathbf{S}$ is acyclic. Then there is a continuous map $\phi:M\to M$ that is homotopic to the identity map on $M$ and such that for every $p\in M$, we have $\phi^{-1}(p)\in \mathbf{S}$. Moreover  $\mathbf{S}=\{\phi^{-1}(p)|p\in M\}$.
\end{thm}
We call the map $M\to M/\sim$ the \emph{Moore map} where $x\sim y$ if and only if $x,y\in S$ for some $S\in \mathbf{S}$. The following definition is \cite[Definition 3.1]{Mar}
\begin{defn}[Admissible decomposition]
Let $\mathbf{S}$ be an upper semi-continuous decomposition of $M$. Let $G$ be a subgroup of $\Homeo(M)$. We say that $\mathbf{S}$ is admissible for the group $G$ if the following holds:
\begin{itemize}
\item Each $f\in G$ preserves setwise every element of $\mathbf{S}$.
\item Let $S\in \mathbf{S}$. Then every point, in every frontier component of the surface $M-S$ is a limit of points from $M-S$ which belong to acyclic elements of $\mathbf{S}$. 
\end{itemize}
If $G$ is a cyclic group generated by a homeomorphism $f:M\to M$ we say that $\mathbf{S}$ is an admissible decomposition of $f$.
\end{defn}
An admissible decomposition for $G<\Homeo(M)$ is called \emph{minimal} if it is contained in every admissible decomposition for $G$. We have the following theorem \cite[Theorem 3.1]{Mar}.
\begin{thm}[Existence of minimal decompositions]
Every group $G<\Homeo(M)$ has a unique minimal decomposition.
\end{thm}
Denote by $\mathbf{A}(G)$ the sub collection of acyclic sets from $\mathbf{S}(G)$. By a mild abuse of notation, we occasionally refer to  $\mathbf{A}(G)$ as a subset of $S_g$ (the union of all sets from $\mathbf{A}(G)$). To distinguish the two notions we do the following. When we refer to $\mathbf{A}(G)$ as a collection then we consider it as the collection of acyclic sets. When we refer to as a set (or a subsurface of $S_g$) we have in mind the other meaning.

We have the following result \cite[Proposition 2.1]{Mar}.
\begin{prop}\label{subsurface}
Every connected component of $\mathbf{A}(G)$ (as a subset of $S_g$)   is a subsurface of $M$ with finitely many ends. 
\end{prop}
\begin{lem}\label{finer}
For $H<G<\Homeo(M)$, we have that  $\mathbf{A}(G)\subset \mathbf{A}(H)$.\end{lem}
\begin{proof}  The inclusion $\mathbf{A}(G)\subset \mathbf{A}(H)$ is because that the minimal decomposition of $G$ is also an admissible decomposition of $H$ and the minimal decomposition of $H$ is finer than that of $G$. 
\end{proof}

\subsection{Lifting through hyper-ellpitic branched cover}
Denote by $S_{g;n,b}$ the surface of genus $g$ with $b$ boundary components and $n$ marked points. To make the analysis easier, we take the following hyper-elliptic $\mathbb{Z}/2$ branched covers 
\[
\pi_n: S=S_{\frac{n-1}{2};n,1}\to S_{0;n,1}    \text{ for $n$ odd or } \pi_n:S=S_{\frac{n}{2}-1;n,2}\to S_{0;n}\text{ for $n$ even}.\]
The cover is shown by the following figures. The hyperelliptic involution on $S$ is denoted by $\tau$.

\begin{figure}[H]
\minipage{0.45\textwidth}
  \includegraphics[width=\linewidth]{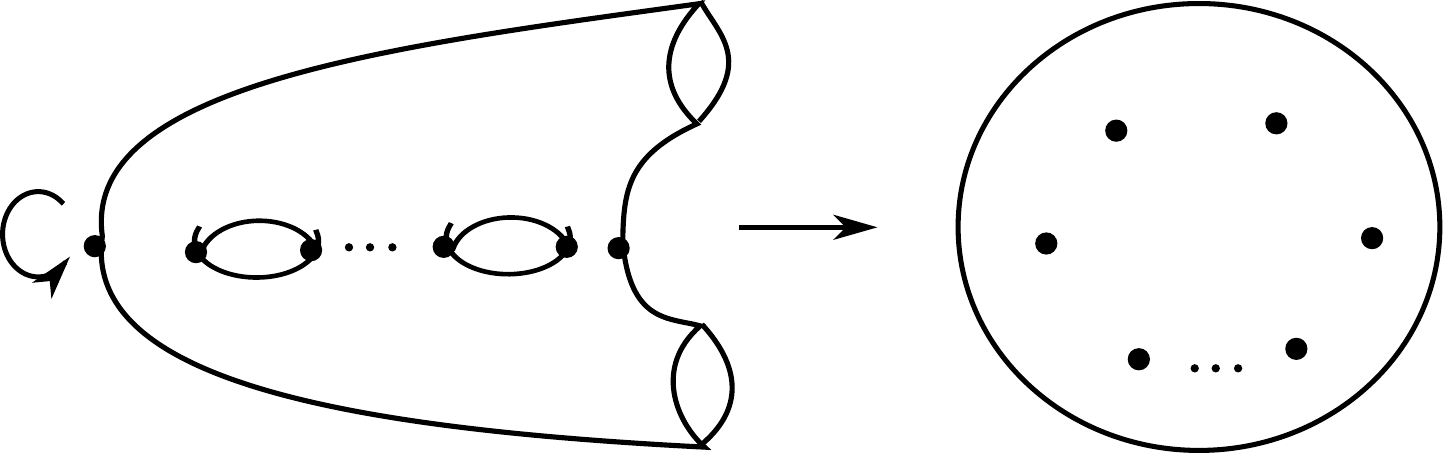}
  \caption{$n$ even}
  \label{p1}
\endminipage\hfill
\minipage{0.45\textwidth}
  \includegraphics[width=\linewidth]{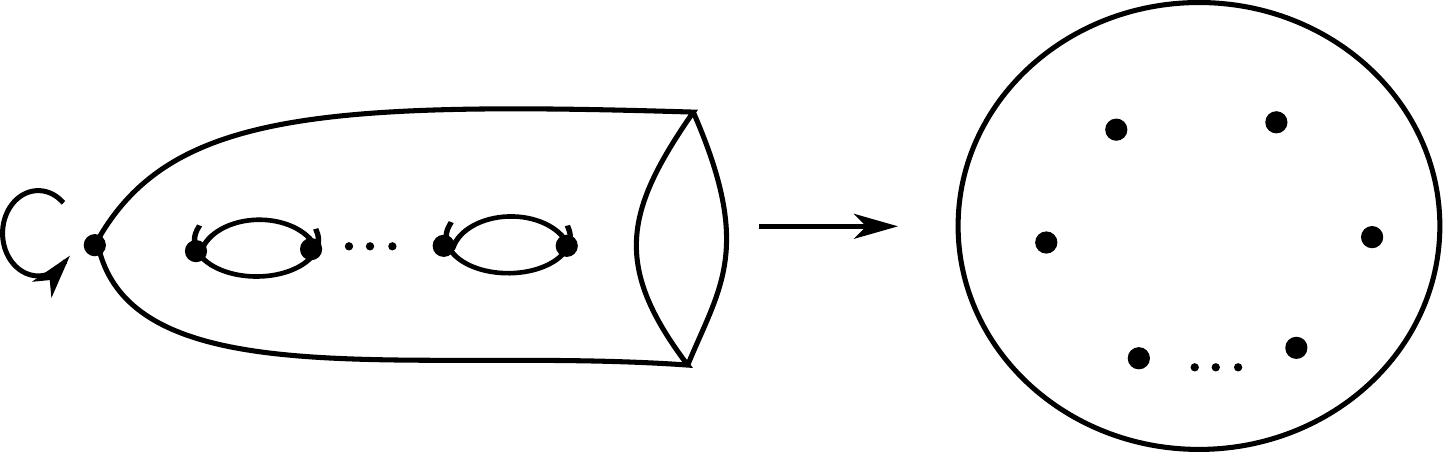}
  \caption{$n$ odd}
  \label{p2}
\endminipage\hfill
\end{figure}

Denote by $\widetilde{PB}_n$ the lifts of mapping classes under $\pi_n$, where it satisfies the following
\[
1\to \mathbb{Z}/2\to \widetilde{PB}_n\xrightarrow{L} PB_n\to 1.
\]
Let $c$ be a simple closed curve on $S_{0;n,1}$ and denote by $T_c$ the Dehn twist about $c$. For every simple closed curve $c$ on $S_{0;n,1}$, we have the following easy fact about its preimage under $\pi_n$.
\begin{fact}
\begin{enumerate}
\item If $c$ bounds odd number of points, then the lift is a single curve $c'$. The preimage of $T_c^2$ under $L$ are $T_{c'}$ and $T_{c'}\tau$. 
\item If $c$ bounds even number of points, then the lift is two curves $c_1,c_2$. The preimage of $T_c$ under $L$ are $T_{c_1}T_{c_2}$ and $T_{c_1}T_{c_2}\tau$. In particular, if  $c$ bounds $2$ points, then $c_1=c_2$.
\end{enumerate} 
\end{fact}
From the above fact, we know that if $c$ bounds $2$ points and $c_1=c_2$ are the lifts, we have that $T_{c_1}^2\in \widetilde{PB}_n$. We have the following.
\begin{fact}
If $\alpha$ is a nonseparating simple closed curve that is invariant under $\tau$, then a square of the Dehn twist about $c$ is in $\widetilde{PB}_n$. We call such element an \emph{invariant Dehn twist square}. \end{fact}

Let $b$ be the curve in $D_n^2$ bounding $5$ points $P_1,...,P_5$. The lift of $b$ under the cover $\pi_n$ is a curve $c$ bounding a genus $2$ subsurface as the following figure. 
\begin{figure}[H]
  \includegraphics[width=3in]{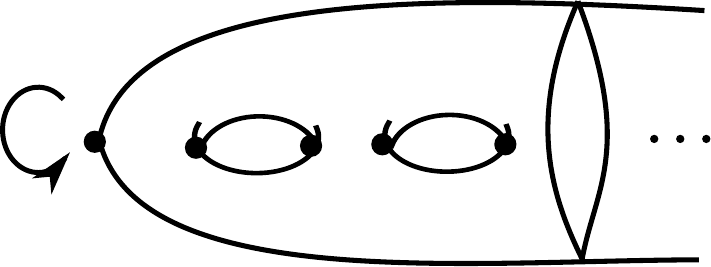}
  \caption{the curve $c$ bounding a genus $2$ surface is the lift of a curve bounding $5$ points}
  \label{p3}
\end{figure}
If a curve $\alpha$ is on the genus $2$ subsurface of $S$ that is cut out by $c$, then we call the invariant Dehn twist square about $\alpha$ a \emph{left invariant Dehn twist square}. We have the following important relations in $\widetilde{PB}_n$. 
\begin{prop}\label{relation}
The element $T_c\in \widetilde{PB}_n$ is a product of left invariant Dehn twist squares in $\widetilde{PB}_n$. 
\end{prop}
\begin{proof}
We have the basic fact that $T_b$ is generated by Dehn twists about curves in the interior of $b$ bounding $2$ points; see, e.g., \cite[Chapter 9]{FM}. Take a lift of all of the elements, we obtain that a product of squares of Dehn twists about nonseparating curves that are disjoint from $c$ and on the left of $c$ in $\widetilde{PB}_n$. After taking the square of the equation, we obtain the proposition.
\end{proof}

\subsection{Characteristic annuli}

From now on, we work with the assumption that there exists a realization of the pure braid group 
\[
\mathcal{E}': PB_n\to \Homeo^a_+(D^2_n).
\]
Lifting by the hyperelliptic involution, we obtain a new realization
\[
\mathcal{E}: \widetilde{PB}_n\to \Homeo^a_+(S_g)^\tau
\]
where the image lies in the centralizer of the hyper-elliptic involution $\tau$. We now only work with the new realization $\mathcal{E}$.

For an element $f\in \widetilde{PB}_n$, or a subgroup $F<\widetilde{PB}_n$, we shorten $\mathbf{A}(\mathcal{E}(f))$ as $\mathbf{A}(f)$, and $\mathbf{A}(\mathcal{E}(F))$ as $\mathbf{A}(F)$, to denote the corresponding collections of acyclic components. Denote by $S$ the hyper-elliptic cover we defined in Section 3.2. Recall that $c\subset S$ is a separating curve that is invariant under $\tau$ that divides $S$ into subsurfaces $S_L$ of genus $2$ and $S_R=S-S_L$ (see more about $c$ in the previous section). We know that $T_c\in \widetilde{PB}_n$. We have the following theorem about the minimal decomposition of $\mathcal{E}(T_c)$.

\vskip .3cm

\begin{thm}\label{minimal}
The set $\mathbf{A}(T_c)$ has a component $\mathbf{L}(c)$ which is homotopic to $S_L$ and a component $\mathbf{R}(c)$ homotopic to $S_R$.
\end{thm}
\begin{proof}[Proof sketch]
The proof is the same as the proof of \cite[Theorem 4.1]{ChenMark}. We use the fact that there are pseudo-Anosov elements on the left and on the right of $c$ in $\widetilde{PB}_n$. In this theorem, we need $n\ge 9$.
\end{proof}

For the rest of paper, denote by 
\[
\mathbf{B}:=S-\mathbf{L}(c)-\mathbf{R}(c).
\]
Let $p_L: \mathbf{L}(c)\to \mathbf{L}(c)/\sim$ and $p_R: \mathbf{R}(c)\to \mathbf{R}(c)/\sim$ be the Moore maps of $\mathbf{L}(c)$ and $\mathbf{R}(c)$ corresponding to the decomposition $\mathbf{S}(c)$. Let $\mathbf{L}\subset \mathbf{L}(c)/\sim$ be an open annulus bounded by the end of $\mathbf{L}(c)'$ on one side, and by a simple closed curve on the other. The open annulus $\mathbf{R}\subset \mathbf{R}(c)/\sim$ is defined similarly. We have the following definition (see \cite[Chapter 5]{Mar}). 

\begin{defn} An annulus of the form $A=p_L^{-1}(\mathbf{L})\cup \mathbf{B} \cup p_R^{-1}(\mathbf{R})$ is called  a \emph{characteristic annulus}.
\end{defn}

\vskip .3cm

Denote by $f=\mathcal{E}(T_c)$. Every characteristic annulus is invariant under $f$. We observe that $\mathbf{B}$ is a separator in $A$, that is, $\mathbf{B}$ is an essential, compact, and connected subset of $A$.  Note that a characteristic annulus $A$ is invariant under $f$, but it may not be invariant under homeomorphisms which are lifts 
(with respect to $\mathcal{E}$) of other elements from $\widetilde{PB}_n$.  However,  $\mathbf{B}$ is invariant under these lifts  of elements from  the image under $\mathcal{E}$ of the centralizer of $T_c$ in $\widetilde{PB}_n$. As we see from the next lemma, the dynamical information about $f$ is contained in $\mathbf{B}$.

\vskip .3cm

\begin{lem}\label{gap} Fix a characteristic annulus $A$. Then
\begin{enumerate}
\item  every number $0<r<1$ appears as the rotation number $\rho(f,x,A)$, for some $x\in A$,
\item if  $0<\rho(f,x,A)<1$, then $x\in B$.
\end{enumerate}
\end{lem}

The proof of the above lemma can be seen in \cite[Lemma 4.5]{ChenMark}. The reason is that $f$ is homotopic to a Dehn twist and that the realization is area-preserving.

\section{The proof of Theorem \ref{main}}
In this section, we discuss the proof of Theorem \ref{main}. We now discuss the main strategy.

\subsection{Outline of the proof}
Recall that $c$ is a separating simple closed curve that divides the surface $S$ (the hyper-elliptic cover of $S_{0;1,n}$) into a genus $2$ subsurface and the rest.  Fix a characteristic annulus $A$. Let $E_r$ be the set of points in $A$ that have rotation numbers equal to $r$ under $\mathcal{E}(T_c)$. Lemma \ref{gap} states that the set $E_r$ is not empty when $0<r<1$.

The key observation of the proof lies in the analysis of connected components of $E_r$. Let $E$ be a component of $E_r$. We show the following:

\begin{enumerate}

\item $E$ is $\mathcal{E}(h)$-invariant for $h$ a left invariant Dehn twist square,
\vskip .3cm

\item $\overline{E}$ is a separator in $A$,

\vskip .3cm

\item if $E$ contains a periodic orbit, then $E$ contains  a separator. 

\end{enumerate}

Denote by $K(\overline{E})$  the annular completion of $\overline{E}$, and  let $\rho(\mathcal{E}(T_c), K(\overline{E}))$ be the rotation interval of $K(\overline{E})$. We claim that $\rho(\mathcal{E}(T_c), K(\overline{E}))=\{r\}$. First of all, we know that  $r\in \rho(\mathcal{E}(T_c), K(\overline{E}))$. If $\rho(\mathcal{E}(T_c), K(\overline{E}))\neq \{r\}$, then $\rho(\mathcal{E}(T_c), K(\overline{E}))$ contains infinitely many rational numbers. By Theorem \ref{PB}, there exist three periodic points $x_1,x_2,x_3\in K(\overline{E})$ with different rational rotation numbers $r_1,r_2,r_3$. Let $F_i$ denote the connected component of $E_{r_{i}}$ containing $r_i$, and let $M_i\subset F_i$ be a separator. 

By Proposition \ref{ordering}, there is an ordering on disjoint separators. Without loss of generality, we assume that $M_1<M_2<M_3$. Based on a discussion about the position $E$ with respect to $M_i$'s, we obtain a contradiction.
Thus,  $\rho(\mathcal{E}(T_c), K(E))$ is the singleton $\{r\}$. 

We know from Theorem \ref{Matsumoto} that the left and right prime ends rotation numbers of $K(\overline{E})$ are both $r$. But in the group of circle homeomorphisms, the centralizer of an irrational rotation is essentially $SO(2)$. 

We then show a new ingredient of the proof: the rotation numbers of the realization of a left invariant Dehn twist square on the set of prime ends of $K(\overline{E})$ are all $0$. This contradicts the fact that  $T_c$ is a product left invariant Dehn twist squares
 as in Proposition \ref{relation}.

\subsection{The set $E_r$} Once again we use abbreviation $f=\mathcal{E}(T_c)$. For a characteristic annulus $A$, we let 
$$
E_r=\{x\in A: \rho\big(f,x,A\big)=r\}. 
$$
By Lemma \ref{gap}, if $0<r<1$, we know that $E_r$ is nonempty and $E_r\subset \mathbf{B}$.

\vskip .3cm

Next, we have the following key lemmas which corresponds to \cite[Lemma 5.1, 5.3, 5.4]{ChenMark}.

\begin{lem}\label{invariant}
Fix $0<r<1$, and let $E$ denote a connected component  of $E_r$. Fix a left invariant Dehn twist square $h$ in $\widetilde{PB}_n$. For $x\in E$, let $C(x) \in \mathbf{A}(h)$ be the corresponding acyclic set. Then $C(x)\subset E$.  In particular,  $E$  is $\mathcal{E}(C(T_c))$-invariant.
\end{lem}

\begin{lem} \label{separatingL0}
The closed set $\overline{E}$ is a separator (as defined in Section 2). 
\end{lem}

\begin{lem}\label{rational}
Let $x$ be a periodic orbit of $f$ such that $\rho(f,x,A)=p/q$ and $0<p/q<1$. Then, the connected component $E$ of $E_{p/q}$ which contains $x$, also contains a separator (as a subset).
\end{lem}

Fix an irrational number $r\in (0,1)$. By Lemma \ref{gap}, we know that $E_r$ is not empty. Let $E$ be a connected component of $E_r$. 
By Lemma \ref{invariant}, we know that $E$ is invariant under $\mathcal{E}(C(T_c))$. By Lemma \ref{separatingL0}, we know that $\overline{E}$ is a separator. The annular completions $K(\overline{E})$ of $\overline{E}$ is also $\mathcal{E}(C(T_c))$-invariant since the definition is canonical. 
The following claim is at the heart  of the entire construction. 
\begin{claim}
Let $r_L$ and $r_R$ be the left and right prime ends rotation numbers of $f$ on $K(\overline{E})$. Then $r_L=r_R=r$. 
\end{claim}

\begin{rem}
We refer the reader to \cite[Claim 5.2]{ChenMark} for the proof. The only property we use about $\widetilde{PB}_n$ is Proposition \ref{relation}.
\end{rem}

\subsection{Finishing the proof}

We need to show a new property of a left invariant Dehn twist square $h \in \widetilde{PB}_n$.
\begin{thm}\label{zero}
The action of $\mathcal{E}(T_b^2)$ on the set of prime ends of $K(\overline{E})$ has rotation number $0$.
\end{thm}
\begin{proof}
Now we consider the rotation set of $\mathcal{E}(T_b^2)$ on $K(\overline{E})$. We claim that the rotation set satisfies
\[
\sigma(\mathcal{E}(T_b),K(\overline{E}))=\{0\}.
\] The reason is that if not, then it is a nontrivial closed interval. By Theorem \ref{PB}, rational rotation numbers are realized by periodic orbit. However $K(\overline{E})\subset B$, that means every point for $x\in  K(\overline{E})\subset B$, there exists $C(x)\in \mathbf{A}(T_b^2)$ such that $C(x)\subset B$ by Lemma \ref{invariant}. However $C(x)$ is acyclic and fixed by $\mathcal{E}(T_b^2)$. Therefore, we know that the rotation number of $\mathcal{E}(T_b^2)$ on points in $C(x)$ is zero, which is a contradiction. Then by Theorem \ref{Matsumoto}, we know that the rotation number of the action of $\mathcal{E}(T_b^2)$ on the set of prime ends is also zero.
\end{proof}

We now finish the proof.

\begin{proof}
Since the rotation number of $\mathcal{E}(T_c)$ on the prime ends of $K(\overline{E})$ is an irrational number $r$, then it is semiconjugate to an irrational rotation. Then up to the same semiconjugation, the image of the centralizer of $T_c$ under $\mathcal{E}$ is $SO(2)$. The image of each element is determined by its rotation number. However, $\mathcal{E}(T_c)$ is a product of $\mathcal{E}(T_b^2)$ for $b$ nonseparating and invariant under $\tau$ by Proposition \ref{relation}. By Lemma \ref{zero}, we know that the rotation number of $\mathcal{E}(T_b^2)$ is zero. Thus their product should also have $0$ rotation number. This contradicts the fact that the rotation number of $\mathcal{E}(T_c)$ is $r$, which is nonzero.
\end{proof}

    	\bibliography{citing}{}
	\bibliographystyle{alpha}

\end{document}